\documentclass{article}

\usepackage{arxiv}

\usepackage[utf8]{inputenc} 
\usepackage[T1]{fontenc}    
\usepackage{hyperref}       
\usepackage{url}            
\usepackage{booktabs}       
\usepackage{amsfonts}       
\usepackage{nicefrac}       
\usepackage{microtype}      
\usepackage{graphicx}
\usepackage{doi}
\usepackage[
color=orange!80, 
bordercolor=black,
textwidth=3cm,
textsize=small,
colorinlistoftodos]
{todonotes}
\usepackage{comment}

\usepackage[
backend=biber,
style=alphabetic,
maxbibnames=99
]{biblatex}
\usepackage{amsthm}
\usepackage{amsmath}
\usepackage{amsfonts}
\usepackage{ amssymb }
\usepackage{xfrac}
\usepackage{quiver}
\usepackage{float}
\usepackage{array}
\usepackage{scalerel}

\newcommand{\cat}[1]{\ensuremath{\mathbf{#1}}}

\newcommand{\Z}{\mathbb Z}
\newcommand{\R}{\mathbb R}

\newcommand{\Q}{\mathbb Q}
\newcommand{\T}{\mathbb T}
\newcommand{\F}{\mathbb F}

\newcommand{\Span}{{\operatorname{Span}}}

\newcommand{\gen}[1]{\left\langle#1\right\rangle}

\theoremstyle{plain}
\newtheorem{theorem}{Theorem}[section]
\newtheorem{corollary}[theorem]{Corollary}
\newtheorem{lemma}[theorem]{Lemma}
\newtheorem{proposition}[theorem]{Proposition}

\theoremstyle{definition}
\newtheorem{definition}[theorem]{Definition}

\newenvironment{conventions}{\par\noindent \textbf{Conventions:} \rmfamily}{\medskip}

\newenvironment{forwardtheorem}[2]{\par\noindent \textbf{#1 #2.} \rmfamily\itshape}{\medskip}

\newcounter{remark}
\newenvironment{remark}{\refstepcounter{theorem}\par
   \noindent \textbf{Remark~\thetheorem.} \rmfamily}{\medskip}

\newcounter{example}
\newenvironment{example}[1][]{\refstepcounter{theorem}\par
   \noindent \textbf{Example~\thetheorem. #1} \rmfamily}{\medskip}

\makeatletter
\def\renewtheorem#1{%
  \expandafter\let\csname#1\endcsname\relax
  \expandafter\let\csname c@#1\endcsname\relax
  \gdef\renewtheorem@envname{#1}
  \renewtheorem@secpar
}
\def\renewtheorem@secpar{\@ifnextchar[{\renewtheorem@numberedlike}{\renewtheorem@nonumberedlike}}
\def\renewtheorem@numberedlike[#1]#2{\newtheorem{\renewtheorem@envname}[#1]{#2}}
\def\renewtheorem@nonumberedlike#1{  
\def\renewtheorem@caption{#1}
\edef\renewtheorem@nowithin{\noexpand\newtheorem{\renewtheorem@envname}{\renewtheorem@caption}}
\renewtheorem@thirdpar
}
\def\renewtheorem@thirdpar{\@ifnextchar[{\renewtheorem@within}{\renewtheorem@nowithin}}
\def\renewtheorem@within[#1]{\renewtheorem@nowithin[#1]}
\makeatother

\makeatletter
\@ifpackageloaded{fdsymbol}\@tempswafalse\@tempswatrue
\if@tempswa
  \newcommand{\fdsy@scale}{1.0}
  \newcommand\fdsy@mweight@normal{Book}
  \newcommand\fdsy@mweight@small{Book}
  \newcommand\fdsy@bweight@normal{Medium}
  \newcommand\fdsy@bweight@small{Medium}
  \DeclareFontFamily{U}{FdSymbolA}{}
  \DeclareSymbolFont{fdsymbols}{U}{FdSymbolA}{m}{n}%
  \SetSymbolFont{symbols}{bold}{U}{FdSymbolA}{b}{n}%
  \DeclareFontShape{U}{FdSymbolA}{m}{n}{
      <-7.1> s * [\fdsy@scale] FdSymbolA-\fdsy@mweight@small
      <7.1-> s * [\fdsy@scale] FdSymbolA-\fdsy@mweight@normal
  }{}
  \DeclareFontShape{U}{FdSymbolA}{b}{n}{
      <-7.1> s * [\fdsy@scale] FdSymbolA-\fdsy@bweight@small
      <7.1-> s * [\fdsy@scale] FdSymbolA-\fdsy@bweight@normal
  }{}
  \DeclareMathSymbol{\aleph}{\mathord}{fdsymbols}{"C7}
  \DeclareMathSymbol{\beth}{\mathord}{fdsymbols}{"C8}
  \DeclareMathSymbol{\gimel}{\mathord}{fdsymbols}{"C9}
  \DeclareMathSymbol{\daleth}{\mathord}{fdsymbols}{"CA}
\fi
\makeatother

\title{A Universal Property of the Spectrum of a Ring and the Semiring of its Ideals}

\author{William Bernardoni\\
	Case Western Reserve University\\
	Cleveland, OH 44118 \\
	\texttt{wrb37@case.edu}
}

\renewcommand{\headeright}{}
\renewcommand{\undertitle}{}

\renewcommand{\shorttitle}{A Universal Property of the Spectrum of a Ring and the Semiring of its Ideals}

\hypersetup{
pdftitle={Universal Property of the Spectrum of a Ring and the Semiring of its Ideals},
pdfsubject={Universal Property},
pdfauthor={William Bernardoni},
pdfkeywords={Valuations (16W60), Non-Archimedean Valued Fields (12J25), Semirings (16Y60), Tropical Geometry (14T10)},
}

\begin{document}
\maketitle

\begin{abstract}
	Arithmetic valuations are intimately connected with the structure of the ideals of a commutative ring. We show how the generalized idempotent semiring valuations of Jeffrey and Noah Giansiracusa can be used to make this connection explicit. Through this generalized valuation theory sufficiently complete positive valuations give rise to Galois correspondences with the lattice of ideals of a commutative ring. Up to isomorphism the semiring of ideals of a commutative ring can be defined as a universal factoring semiring for positive valuations. We then further show that this valuation theory can formally connect Joyal's \textit{notion de z\'eros} and universal property of the spectrum of a ring to arithmetic valuation theory. We show that up to isomorphism as a coframe, the closed Zariski topology on a spectrum of a commutative ring can be defined as the universal factoring semiring of additively and multiplicatively idempotent valuations.
\end{abstract}

\section*{Disclosure Statement}
The author reports there are no relevant financial or non-financial competing interests in the publication of this research.


\section{Introduction}

Arithmetic valuations are a powerful tool connecting algebraic geometry, tropical geometry, and non-Archimedean analysis. They were an integral component of Zariski's conception of algebraic geometry, and are intimately connected with $p$-adic numbers and $p$-adic analysis. Equipping a ring with a valuation allows one to tropicalize, bringing in additional tools to study algebro-geometric objects over the ring, such as the method of patchworking created by Oleg Viro to classify isotropy classes of real algebraic curves of degree 7 \cite{viro2006patchworking, Mikhalkin}.

Saying that a ring admits a non-degenerate valuation is equivalent to placing a non-Archimedean absolute value on the ring. There are many functionals on algebras however that appear similar to standard valuations but are not \textit{exactly} non-Archimedean absolute values. This insight led to generalizations of arithmetic valuations.

One such generalized notion is due to Krull and such valuations are called \textit{Krull valuations}, wherein the codomain of the valuation was taken to be a totally ordered abelian semigroup rather than strictly the extended real numbers. Zariski used and extended this notion to form a basis for his algebraic geometry.

In their paper on a scheme theoretic tropicalization of toric varieties, \cite{2016}, Noah and Jeffrey Giansiracusa introduced a further generalization of a valuation, where the codomain was taken to be an idempotent semiring. This allows one to consider \textit{partially ordered} valuations and extend the idea of tropicalization to further semirings. This approach can also be used to define a \textit{non-commutative} valuation theory, which was partially explored in \cite{bernardoni2023applications}. This generalization allows one to connect the standard theory of arithmetic valuations with similar notions, such as Big-O asymptotic complexity in computer science, and assigning zero sets to polynomials.

The valuations on a ring are strongly connected with the underyling structures of the ring. Recent work by Noah and Jeffrey Giansiracusa have shown that in a scheme-theoretic setting their generalized notion of a valuation is intimately connected with the Berkovich analytification \cite{giansiracusa2022universal}. Classical results, such as Krull's classification of valuation rings \cite{Krull1939}, show that the existence of standard arithmetic valuations is intimately tied to the structure of the lattice of ideals of a given ring.

In this paper we will clarify the connection between the structure of ideals of a commutative ring and the valuations of that ring. We will also show how the generalized valuations that exist on a ring determine and are determined by its prime spectrum. To do so we explore two further specializations of the Giansiracusas' valuation theory and use them to derive defining universal properties for the semiring of ideals of a commutative ring as well as the coframe of the closed Zariski topology on its prime spectrum.

We will first examine positive valuations -- valuations whose codomain are idempotent semirings bounded by $0$ and $1$, and show that the semiring of ideals of a commutative Noetherian ring is the universal semiring that factors all such positive valuations. When we further restrict the codomain to complete idempotent semirings bounded by $0$ and $1$, also known as commutative integral quantales, and show that for any commutative ring its quantale of ideals is defined up to isomorphism as the universal quantale that factors such maps.

As a result we have in Theorem \ref{thm:GaloisConnection} that every sufficiently complete positive valuation induces a Galois connection on the ideals of a commutative ring and Corollary \ref{thm:I(R) initial} stating that the semiring of ideals may be defined up to isomorphism by examining the positive valuations of a commutative ring.

\begin{forwardtheorem}{Theorem}{\ref{thm:GaloisConnection}}
    Let $R$ be a commutative ring and $I(R)$ its semiring of ideals.
    
    Let $\nu : R \to \Gamma$ be a positive valuation, and let $\Gamma$ be a complete semiring. Then there exists a Galois connection between $I(R)$ and $\Gamma$.

    That is, there exist monotone maps $N : I(R) \to \Gamma$ and $I: \Gamma \to I(R)$ such that:
    \[\gamma \le N(J) \iff J \subseteq I(\gamma)\]
    For $\gamma \in \Gamma$, $J \in I(R)$
\end{forwardtheorem}

\begin{forwardtheorem}{Corollary}{\ref{thm:I(R) initial}}
    Let $R$ be a commutative ring.
    
    $I(R)$ is initial among positive valuations whose valuation semiring is complete for sums of size $|R|$ or less.

    If $R$ is Noetherian then $I(R)$ is initial among all positive valuations.
\end{forwardtheorem}

We will then restrict further to multiplicatively idempotent positive valuation -- valuations whose codomain are appropriately bounded semirings that are multiplicatively and additively idempotent. In the Noetherian case the Zariski closed topology on the prime spectrum is the universal semiring that factors such maps. If we restrict to the complete case, and study valuations whose codomain is a coframe (a complete, multiplicative and additively idempotent semiring), then for any commutative ring the Zariski closed topology is the universal coframe that factors all such valuations.

This will culminate in Theorem \ref{thm:SpecisUniversal}.

\begin{forwardtheorem}{Theorem}{\ref{thm:SpecisUniversal}}
    Let $R$ be a commutative ring.
    
    The initial valuation among appropriately complete, multiplicatively idempotent positive valuations is $\tau : R \to T$, where $T$ is the coframe of the closed Zariski topology on the spectrum of a commutative ring and $\tau(r)$ is the collection of prime ideals containing the element $r$.
\end{forwardtheorem}

One of the core results of this paper, that the spectrum of a ring is the ``universal zero locus", is a folkloric result. This result originated in the note by Andr\'e Joyal, \cite{Joyal}, where it was stated that the prime spectrum of a ring factors all \textit{notion de z\'eros} -- valuations whose image is in a distributive lattice. This result also appears in the book, \cite{StoneSpaces} in section V.3.1 where it is used as an equivalent definition of the Zariski spectrum of a ring. This result is mentioned and utilized in a Mathematics Stack Exchange answer by Zhen Lin Low, \cite{Zhen}. While this particular result may not be novel, there is novelty in our approach, and value in explicitly stating this result and proof in terms of valuations of commutative rings. This work fully connects Joyal's notion de z\'eros with classical arithmetic valuation theory.

While this work is focused on the valuation theory of commutative rings, there is recent work developing a valuation theory applied to idempotent semirings themselves. We can find similar results in this other algebraic context.  In the work \cite{tolliver2016extension}, valuations are taken to be any homomorphism between an idempotent semiring and a totally ordered one. The central result in \cite{bernardoni2023applications} can be seen as a ring-theoretic analog of a central result in \cite{tolliver2016extension}. In \cite{tolliver2016extension}, Jeffrey Tolliver finds that the universal valuation on an idempotent semiring is given by the collection of $\Z$-submodules, we see this reflected in the ring case in \cite{bernardoni2023applications} where the universal valuation on a ring is given by the collection of $\Z$-submodules of the ring. The work \cite{jun2021lattices} finds that the space of surjective valuations between idempotent semirings is isomorphic to the prime spectrum of that \textit{semi}ring.

\section*{Outline}
\begin{itemize}
    \item In Section \ref{sec:preliminaries} we introduce the algebraic definitions we will use. Section \ref{sec:semiring} gives definitions and examples of semirings, idempotent semirings, and complete idempotent semirings. A reader familiar with semirings may skip this section, noting that we use the order convention that $a \le b \iff a + b = a$. Section \ref{sec:Valuation} gives the generalized definition of an idempotent semiring valuation in the sense of \cite{2016}.
    \item In Section \ref{sec:UnivVal} we give the definition of a universal valuation (Definition \ref{UnivVal}), and repeat prior results from \cite{bernardoni2023applications} giving a classifying structural theorem (Theorem \ref{thm:valladdstruct}) to these universal valuation semirings. We give an explicit example (Example \ref{ex:SmallestNC}) of a non-trivial, non-commutative universal valuation semiring.
    \item In Section \ref{sec:PostiveVal} we introduce our notion of a \textit{positive} valuation (Definition \ref{def:PositiveVal}), we show that every positive valuation gives rise to a Galois connection with the lattice of ideals of a commutative ring (Theorem \ref{thm:GaloisConnection}) and we show that with minimal assumptions the universal \textit{positive} valuation semiring of a commutative ring $R$ is the semiring of ideals of that ring $I(R)$ (Corollary \ref{thm:I(R) initial}). We then show how one can recover many properties of the ideals of a ring purely from the semiring structure of $I(R)$.
    \item In Section \ref{sec:FrameVal} we restrict further to \textit{multiplicatively idempotent} positive valuations (Definition \ref{def:MultIdempVal}) -- e.g. coframe or set valued valuations -- and we show how the generalized notion of a valuation given by the Giansiracusas can be used to link Joyal's \textit{notion de z\'eros}, \cite{Joyal} and Johnstone's defining universal property of the Zariski spectrum (Section V.3.1 of \cite{StoneSpaces}) of a ring to classical arithemetic valuation theory. To do so we show that the universal multiplicatively idempotent positive valuation semiring of a commutative ring is given by the closed Zariski topology on its prime spectrum (Theorem \ref{thm:SpecisUniversal}).
\end{itemize}

\section*{Acknowledgments}
We would like to thank Nick Gurski for his mentorship, conversations and advice that helped shape this paper. We would also like to thank Zhen Lin Low for his correspondence and help in finding the references to Joyal and Johnstone's work. We would also like to thank Roland Baumann, Stephanie Egler, and Andrew Edwards for great mathematical conversations and feedback.

\section{Algebraic Preliminaries}
\label{sec:preliminaries}

This section is adapted from section 2 of \cite{bernardoni2023applications}.

\begin{conventions}
    \begin{itemize}
        \item In this paper we assume semirings, rings, and monoids are unital.
        \item We denote the powerset of a set $X$ as $2^X$.
        \item In equations that involve multiple algebraic objects, we denote by $+_S$ and $*_S$ to be the addition and multiplication respectively in the object $S$. In particular when we refer to standard arithmetic we will note $+_\R$ and $*_\R$ to refer to the standard addition and multiplication over the real numbers.
        \item In general we will use $R$ to designate an arbitrary ring, $S$ an arbitrary semiring, and $\Gamma$ an arbitrary idempotent semiring.
    \end{itemize}
\end{conventions}

\subsection{Semirings}\label{sec:semiring}
In this paper we will be concerned with valuations whose codomain are idempotent semirings.

\begin{definition}\label{def:semiring}
A \textbf{semiring} $(S, +, *, 0_S, 1_S)$ is a tuple where $S$ is a set, $+$ and $*$ are binary operations on $S$, and $0_S$ and $1_S$ are elements of $S$ such that:
\begin{enumerate}
    \item $(S,+)$ is a commutative monoid with identity element $0_S$.
    \item $(S,*)$ is a monoid with identity element $1_S$.
    \item Multiplication distributes over addition.
    \item Multiplication by $0_S$ annihilates $R$, that is $0_S * s  = 0_S = s * 0_S$ for all $s \in S$.
\end{enumerate}
\end{definition}

\begin{remark}
    A ring is a semiring whose additive monoid is a \textit{group} and not just a monoid. Semirings can be seen as ``rings without \textbf{n}egation", and so they are sometimes called \textbf{rigs}.
    
    The fact that multiplication by $0$ annihilates in a ring can be proven from the group structure of addition. In a semiring we cannot subtract, and so this must be taken as an axiom.
\end{remark}

\begin{definition}\label{def:idempotent}
A semiring $S$ is called \textbf{idempotent} if for all $a \in S$: $a + a = a$.
\end{definition}

The additive idempotency of a semiring is entirely determined by the idempotency of its multiplicative unit.

\begin{proposition}\label{thm:idemp1}
    A semiring $S$ is idempotent if and only if $1_S + 1_S = 1_S$.
\end{proposition}

\begin{definition}
    We refer to the category of idempotent semirings as $\cat{ISR}$.

    An object in $\cat{ISR}$ is an idempotent semiring, and a morphism $\phi : S \to {S'}$ is a unital semiring homomorphism. That is for all $a, b \in S$:
    \begin{itemize}
        \item $\phi(a +_S b) = \phi(a) +_{S'} \phi(b),$
        \item $\phi(a *_S b) = \phi(a) *_{S'} \phi(b),$
        \item $\phi(0_S) = 0_{S'},$
        \item $\phi(1_S) = 1_{S'}.$
    \end{itemize}
\end{definition}

\begin{definition}
A semiring $S$ is called \textbf{commutative} if for all $a, b \in S$: $ab = ba$.
\end{definition}

We will chiefly be concerned with commutative idempotent semirings in this paper.

\begin{example}
    The most famous example of a commutative idempotent semiring is the \textbf{tropical semiring} 
    \[\T = (\R \cup \{\infty\}, \min, +_\R, \infty, 0).\]
    The tropical numbers are the extended real numbers with addition being the minimum of two numbers and multiplication being standard addition.
    \[a +_\T b = \min(a,b),\]
    \[a *_\T b = a +_\R b.\]
    The tropical semiring is an important structure in computer science as well as in algebraic geometry. Some good introductions to the applications of the tropical semiring and its geometry are the books \cite{MaclaganSturmfels, Mikhalkin, MaxPlus}.
\end{example}

\begin{remark}
    There are three equivalent formulations of the tropical semiring. The one we use in this paper is the ($\min$,+) tropical semiring as it is the formulation that appears in the standard definition of an arithmetic valuation.

    The other two forms are the $(\max, +)$ tropical semiring: $(\R \cup \{-\infty\},\max, +_{\R}, -\infty, 0)$ and the $(\max, \times)$ tropical semiring: $([0,\infty), \max, \times_{\R}, 0, 1)$.

    The $(\max, +)$ and the $(\min, +)$ semirings are isomorphic via the map
    \[x \mapsto -x.\]

    The $(\max, +)$ and the $(\max, \times)$ semiring are isomorphic via the map
    \[x \mapsto e^{x}.\]
\end{remark}

\begin{example}
    Another core example of a commutative idempotent semiring is the \textbf{boolean semiring}
    \[\mathbb B = (\{\bot, \top\}, \vee, \wedge, \bot, \top)\]
    where $\bot$ represents the logical false, $\top$ the logical true, $\vee$ is logical OR and $\wedge$ is logical AND.
\end{example}

\begin{remark}
    The boolean semiring is the unique idempotent semiring of order $2$ and is the initial object in $\cat{ISR}$.
\end{remark}

\begin{remark}
    All semirings can be seen as $\mathbb N$ algebras. Rings are exactly the semirings that have a $\mathbb Z$ algebra structure, and idempotent semirings are exactly the semirings that admit a $\mathbb B$ algebra structure.
\end{remark}

\begin{definition}\label{def:SemiringOrder}
    Idempotent semirings carry a natural partial order, where we define
\[a \le b \iff a + b = a.\]
    We say $a < b$ if $a \le b$ and $a \ne b$.
\end{definition}
This is in analogy with the ordering on the tropical semiring, $\T = (\R \cup \{\infty\}, \min, +, \infty, 0)$. A useful trick to remember this convention is that ``the baker eats the bigger baguette," as, if $b \ge a$, then $b$ ``is eaten" in the sum. Some authors prefer the flipped convention where $a \ge b$ if and only if $a + b = a$.

\begin{example}
    In $\mathbb B$ we have that $\top < \bot$.

    In $\T = (\R \cup \{\infty\}, \min, +_\R, \infty, 0)$, $a < b$ if $a$ is less than $b$ as elements of $\R$.

    If we instead use the $(\max, +)$ semiring $(\R \cup \{-\infty\}, \max, +_\R, -\infty, 0)$, $a < b$ in our semiring if $a$ is greater than $b$ as elements of $\R$.
\end{example}

Because this ordering is defined algebraically, we have the useful fact that idempotent semiring homomorphisms are monotone with respect to this ordering.

\begin{proposition}
    Semiring homomorphisms between idempotent semirings are order preserving.
\end{proposition}
\begin{proof}
    Let $\phi: S \to S'$ be a semiring homomorphism, and let $a \le b$, i.e. $a + b = a$.
    
    \[\phi(a) + \phi(b) = \phi(a + b) = \phi(a)\]
    i.e. $\phi(a) \le \phi(b)$.
\end{proof}

\begin{definition}
    Let $X \subseteq S$ be a finite subset of an idempotent semiring $S$. We define:
    \[\inf(X) = \sum_{x \in X} x\]

    In particular, for $a, b \in S$, we define
    \[\inf(a,b) = a + b\]
\end{definition}
\begin{remark}
    It is equivalent to define an idempotent semiring as a unital monoid $(S,*, 1_S)$ equipped with a lower semi-lattice order and a maximal element $0_S$. The above definition gives this correspondence.
\end{remark}

\begin{definition}
    If the natural partial order on an idempotent semiring $S$ is a total order then we say that $S$ is a \textbf{totally ordered} idempotent semiring.
\end{definition}

Most idempotent semirings are not totally ordered. For instance, in general the power set of a monoid carries a non-totally ordered semiring structure:
\begin{example}\label{ex:powerset}
    Let $(M, *)$ be a monoid, not necessarily commutative, with identity $e$. We can construct the \textbf{power set} idempotent semiring of $M$ as the set of subsets of $M$, $2^M$, with addition being union and multiplication being Minkowski multiplication. That is:
    \[A + B = A \cup B\]
    \[A * B = \{ab : a \in A, b \in B\}\]
    The additive identity of $2^M$ is the empty set, $\emptyset$, and the multiplicative identity is the singleton, $\{e\}$.

    $A \le B$ if and only if $A \supseteq B$.\footnote{When dealing with idempotent semirings constructed from families of sets with $+ = \cup$ we will often have to deal with this reversal of order. If we chose the opposite order convention then our partial order would appear flipped when dealing with standard valuations. The ordering is useful shorthand and a useful property, but any choice of order convention will make working with some semirings awkward.}
\end{example}

\begin{proposition}
    $2^M$ is totally ordered if and only if $M = \{e\}$.
\end{proposition}

\begin{definition}
    A semiring $S$ is \textbf{complete} if it has infinitary addition and infinitary distribution. i.e. for any, potentially infinite, index set $I$ there is a summation $\sum_{I}$ such that
    \begin{align*}
        a * \left(\sum_{i\in I} b_i\right) &= \sum_{i \in I} a * b_i\\
        \left(\sum_{i \in I} b_i\right) * a &= \sum_{i \in I} b_i * a.
    \end{align*}

    We require these operations to behave as one would expect with regard to decompositions of the index sets.
    \[\sum_{\emptyset}b_i = 0_S \qquad \sum_{\{1\}} b_i = b_1 \qquad \sum_{\{1,2\}} b_i = b_1 + b_2\]
    If $\bigcup I_j = I$ and the $I_j$ are pairwise disjoint then
    \[\sum_{j \in J}\sum_{i \in I_j} b_i = \sum_{i \in I} b_i.\]

    The above requirements can be stated concisely as: \textit{$S$ is a complete semiring if and only if $(S,+)$ is a complete monoid.}

    Let $\beth$ be a cardinal. $S$ is \textbf{complete for sums of size $\beth$ or less} or \textbf{$\beth$-complete} if it has summation operations that satisfies the corresponding distributive and decomposition laws for any index set $I$ such that $|I| \le \beth$.

    A \textbf{homomorphism of complete semirings} is a semiring homomorphism that respects these summation operations, that is:
    \[\phi(\sum_{I} b_i) = \sum_{I} \phi(b_i)\]
\end{definition}

The power set semiring is a canonical example of a complete idempotent semiring. The boolean and tropical semirings are also complete. All semirings are complete for finite sums, and all idempotent semirings may be extended to be complete (see \cite{goldstern2002completion}).

With respect to the semiring ordering, a complete semiring is not necessarily a complete lattice. We are only requiring it to have and respect potentially infinite infinums, but it does not need to have or respect potentially infinite supremums. In particular even if such supremums exist they may not be preserved by complete semiring homomorphisms.

The complete subsemirings of the lattice of the power set of a set have an important name -- \textit{topologies}.

\begin{example}
    Let $Z$ be a set.
    
    Let $\tau \subseteq 2^Z$ be an \textit{open}-topology. For any collection $X \subseteq \tau$, $\bigcup_{x \in X} x \in \tau$ and for any \textit{finite} collection $\bigcap_{i =1}^n x_i \in \tau$, as well as $\emptyset, Z \in \tau$.

    The topology $\tau$ is a subobject of the complete semiring $2^R$, where for $A, B \in \tau$
    \begin{align*}
        A +_{\tau} B &= A \cup B,\\
        A *_{\tau} B &= A \cap B,\\
        0_{\tau} &= \emptyset,\\
        1_{\tau} &= Z.
    \end{align*}

    For a \textit{closed}-topology $\tau'$ (i.e. closed under infinite intersection and finite unions) we flip these operations with
    \begin{align*}
        A +_{\tau'} B &= A \cap B,\\
        A *_{\tau'} B &= A \cup B,\\
        0_{\tau'} &= Z,\\
        1_{\tau'} &= \emptyset.
    \end{align*}

    In both cases we have a complete semiring. By definition topologies are exactly the complete subsemirings of $2^Z$, closed topologies with respect to the ordering $A \le B \iff A \subseteq B$ and open topologies with respect to the dual ordering $A \le B \iff A \supseteq B$.
\end{example}

\begin{remark}
    Due to this fact complete, commutative, multiplicatively and additively idempotent semirings appear often in topology under the name of \textbf{frames} and \textbf{coframes}. A frame is a generalization of the lattice of open sets of a topology and a coframe is a generalization of the lattice of closed sets of a topology. Frames are the basic structure of the field of point-free topology and a powerful generalization of topological spaces. An introduction to point-free topology and the applications of frames is the book \cite{StoneSpaces}.
\end{remark}

\begin{definition}
    A \textbf{congruence} $C$ of a semiring $S$ is a subsemiring of $S \times S$ that is also an equivalence relation.
    \begin{itemize}
        \item For all $a \in S$, $(a,a) \in C$.
        \item If $(a,b), (b,c) \in C$ then $(a,c) \in C$.
        \item If $(a,b), (c,d) \in C$ then $(b,a), (a+c,b+d), (ac,bd) \in C$.
    \end{itemize}
\end{definition}

Congruences are closed under arbitrary intersection, which gives us the following definition:
\begin{definition}
    Given a set of relations $X \subseteq S \times S$, we say that the congruence generated by $X$, $\gen{X}$, is the intersection of all congruences that contain $X$.
\end{definition}

If our semiring $S$ is a ring then the set of congruences is in bijection with the set of ideals, as each congruence is determined by the equivalence class of $0$.

Congruences are exactly the equivalence relations on semirings such that the equivalence classes form a well defined semiring under the induced operations.
\begin{definition}
    Given a semiring $S$ and a congruence $C$, we say the \textbf{quotient} $S/C$ is the semiring on the equivalence classes of $C$ equipped with the operations induced from $S$.
\end{definition}

A more detailed treatment of congruences and quotient semirings can be found in chapter 8 of \cite{golan1999semirings}.

It is important to note that for general semirings the quotients of a semiring $S$ are not determined by the equivalence class of $0$. There are often many nontrivial quotients of a semiring such that the equivalence class of $0$ is just $\{0\}$.

\begin{definition}
    Given a semiring $S$, and a set of variables $X$, we say that the \textbf{semiring of expressions over $X$}, denoted $S\gen{X}$, and also called the \textbf{non-commutative polynomial semiring}, is the freely generated semiring over elements in $S$ and variables in $X$ with no relations between them apart from defining $1_S = 1_{S \gen{X}}$ and $0_S = 0_{S \gen{X}}$.

    So for instance we have elements in $\T\gen{x,y}$ of the form:
    \[3xy + 12x + x12 + 0\]
    where $12x \ne x12$ and $xy \ne yx$.
\end{definition}

\begin{definition}
    The \textbf{polynomial semiring} over variables $X$ is the quotient
    \[S\gen{X}/\sim\]
    where $\sim$ is the congruence generated by relations specifying that variables from $X$ commute with one another and with elements in $S$. These relations are
    \begin{itemize}
        \item $xy \sim yx$ for $x, y \in X$,
        \item $sx \sim xs$ for $s \in S, x \in X$.
    \end{itemize}

    We denote this semiring $S[X]$.
\end{definition}

$S\gen{X}$ is the semiring of algebraic expressions where we do not assume that our variables commute with $S$ or with each other, whereas $S[X]$ is the semiring where we assume that our variables \textit{do} commute with $S$ and with each other.

\begin{proposition}
    Let $\Gamma$ be an idempotent semiring, $X$ any set, and $C$ any congruence on $\Gamma$. The following are all idempotent semirings.
    \begin{itemize}
        \item $\Gamma \gen {X}$
        \item $\Gamma[X]$
        \item $C$
        \item $\Gamma/C$
    \end{itemize}
\end{proposition}
This follows from Proposition \ref{thm:idemp1}, as $1_\Gamma + 1_\Gamma = 1_S$ for each of these semirings.

A detailed reference on the general theory of semirings is the book \cite{golan1999semirings}.

\subsection{Valuations}\label{sec:Valuation}
\begin{definition}\label{def:valuation}
Let $R$ be a ring and $\Gamma$ an idempotent semiring. We say that a function $\nu : R \to \Gamma$ is a \textbf{valuation} if $\nu$ is:
\begin{description}
    \item[Unital:] $\nu(0_R) = 0_\Gamma$, $\nu(1_R) = 1_\Gamma$,
    \item[Multiplicative:] $\nu(a*_Rb) = \nu(a)*_{\Gamma}\nu(b)$,
    \item[Superadditive:] $\nu(a +_R b) \ge \nu(a) +_{\Gamma} \nu(b) = \inf_{\Gamma}(\nu(a), \nu(b))$.
    \item[Invariant under negation:] $\nu(-1_R) = \nu(1_R)$
\end{description}
\end{definition}

From Definition \ref{def:SemiringOrder} we get that superadditivity is equivalent to the following identity:
\[\nu(a +_R b) +_\Gamma \nu(a) +_{\Gamma} \nu(b) = \nu(a) +_{\Gamma} \nu(b). \]

The above definition aligns with the generalized semiring valuation in \cite{2016}.

\begin{remark}
    You can recover the traditional definition of a valuation in arithmetic number theory by taking $\Gamma = \mathbb T$. A Krull valuation is a valuation where that the canonical order on $\Gamma$ is a total order and $\Gamma$ is commutative.
\end{remark}

\begin{example}
Let $\Gamma = (\Z \cup \{\infty\}, \min, +_\Z, \infty, 0)$.

Fix a prime number $p$. Each $x \in \Q$ can be written in the form $p^n * \frac{a}{b}$ for integers $a, b$ and $n$, where $p$ does not divide $a$ or $b$. The \textbf{$p$-adic valuation on $\mathbb Q$} is the map:
\[\nu_p(x) = n\]
$\nu_p: \Q \to \Gamma$ forms a valuation.
\end{example}

The following lemma also appears in \cite{2016} as Lemma 2.5.3:
\begin{lemma}\label{thm:meetofsum}
Let $R$ be a ring, and let $\nu: R \to \Gamma$ be a valuation on $R$.

For any $a, b \in R$ the following three values are equal: 
\begin{itemize}
    \item $\nu(a) + \nu(b)$
    \item $\nu(a+b) + \nu(a)$
    \item $\nu(a+b) + \nu(b)$
\end{itemize}
\end{lemma}

\begin{remark}
    The above theorem is not just a consequence of superadditivity, but is in fact equivalent to superadditivity. 
    
    This theorem is the generalized version of a common theorem on arithmetic and Krull valuations. If $\nu$ is a Krull valuation and $\nu(a) \ne \nu(b)$ then $\nu(a + b) = \min(\nu(a), \nu(b))$. In the case where $\Gamma$ is totally ordered this is a corollary of the above theorem, however for a non-totally ordered codomain the above statement does not imply that $\nu(a + b) = \min(\nu(a),\nu(n))$ instead we just have that the infinum of any two elements of $\{\nu(a), \nu(b), \nu(a + b)\}$ are the same. 
\end{remark}

\begin{corollary}
    Let $a$ and $b$ be rational numbers with $a \ge b$.
    \[\operatorname{gcd}(a,b) = \operatorname{gcd}(a - b, b)\]
\end{corollary}
\begin{proof}
    To prove this we just note that we have a valuation from the ring $\Q$ into the idempotent semiring $(\Q_{\ge 0}, \operatorname{gcd}, *, 0, 1)$:\footnote{Here we define that for $\frac{a}{b}, \frac{c}{d} \in \Q_{\ge 0}$: $\operatorname{gcd}(\frac{a}{b}, \frac{c}{d}) = \frac{\operatorname{gcd}(ad, bc)}{bd}$ and $\operatorname{gcd}(0,x) = x$ for all $x \in \Q_{\ge 0}$. One can check that the above definition does not depend on a choice of representation for our fractions.}
    \[x \mapsto |x|\]
    Lemma \ref{thm:meetofsum} tells us:
    \[\operatorname{gcd}(|a-b|,|-b|) = \operatorname{gcd}(|a|,|-b|).\]
\end{proof}

\begin{remark}
    We note that due to our invariance under negation and Lemma \ref{thm:meetofsum}, we can also state:
    \[\nu(a - b) + \nu(b) = \nu(a) + \nu(-b) = \nu(a) + \nu(b).\]
\end{remark}

\section{Universal Valuations}\label{sec:UnivVal}

For a fixed ring $R$ we can examine the slice category $R \downarrow \cat{ISR}$ where objects are valuations $\nu : R \to \Gamma$ and arrows are given by semiring homomorphisms that make the following triangle commute.
\[\begin{tikzcd}[cramped]
	& R \\
	\Gamma && {\Gamma'}
	\arrow["\nu"', from=1-2, to=2-1]
	\arrow["{\nu'}", from=1-2, to=2-3]
	\arrow["\varphi"', from=2-1, to=2-3]
\end{tikzcd}\]

Proposition 2.5.4 from \cite{2016} can be stated as the following.
\begin{proposition}
    $R \downarrow \cat{ISR}$ has an initial object, $R \to \Gamma_R$. 
    
    As $R \to \Gamma_R$ is initial in $R \downarrow \cat{ISR}$, it factors all other valuations of $R$. For any valuation $\nu: R \to \Gamma$ there exists a unique semiring homomorphism $\Gamma_R \to \Gamma$ such that the composite $R \to \Gamma_R \to \Gamma$ is $\nu$.

    That is, $\cat{ISR}(\Gamma_R, \Gamma')$ is in bijection with the collection of generalized valuations on $R$ with codomain $\Gamma'$.
\end{proposition}

\begin{remark}
    For the rest of this paper we will often consider valuations over complete idempotent semirings, and consider universal valuation semirings that are universal over valuations whose codomain are complete. Here we can take advantage that we can freely complete semirings, \cite{goldstern2002completion}, and that this completion functor is a left adjoint and so it preserves colimits. In particular this means that the universal complete valuation semiring is the completion of the standard universal valuation semiring. 
\end{remark}

\cite{2016} gives a construction of $\Gamma_R$, the definition we give below is the one in \cite{bernardoni2023applications} that is a very slight modification to apply to potentially non-commutative rings.\footnote{The form in \cite{2016} is the same quotient, but the free semiring is the polynomial semiring $\mathbb B[x_R]$ rather than $\mathbb B\gen{x_R}$. $\mathbb B\gen{x_R}$ is the free idempotent semiring on the set $R$ whereas $\mathbb B[x_R]$ is the free \textit{commutative} idempotent semiring on the set $R$. If $R$ is commutative the quotient forces $\Gamma_R$ to be as well, so when considering commutative rings one may restrict to the category of commutative idempotent semirings.}
\begin{definition}\label{UnivVal}\index{Universal Valuation}
    We construct the \textbf{universal valuation} semiring of $R$, $\Gamma_R$ as follows.

    For each $a \in R$, we define a variable $x_a$. We denote the free semiring of boolean expressions over these variables as
    \[\mathbb B\gen{x_R}\]

    We then define $\Gamma_R$\index{$\Gamma_R$} as
    \[\Gamma_R = \mathbb B\gen{x_R}/\sim\]
    where $\sim$ is the congruence generated by the following relations.
    \begin{align*}
        x_{0} &\sim 0\\
        x_1 &\sim 1 \\
        x_{-1} &\sim 1\\
        x_{a}x_{b} &\sim x_{ab}\\
        x_{a + b} + x_a + x_b &\sim x_{a} + x_b
    \end{align*}

    The valuation associated with this semiring, is the map
    \[\nu(a) = [x_a].\]
\end{definition}

\begin{example}\label{ex:SmallestNC}
    This is example 8 from \cite{bernardoni2023applications}.
    
    Let $R$ be the ring of upper triangular $2 \times 2$ matrices over $\F_2$. $R$ has eight elements and they are generated by the matrices:
    \[i = \begin{bmatrix}1&0\\0&0\end{bmatrix} \qquad j = \begin{bmatrix}0&1\\0&0\end{bmatrix} \qquad k = \begin{bmatrix}0&0\\0&1\end{bmatrix}\]

    $\Gamma_R$ consists of $\mathbb B$ linear combinations of the elements: $0, 1, x_i, x_j, x_k, x_{i+j}, x_{j + k}, x_{i + j + k}$. Note that $x_{i + k} = x_1 = 1$.

    The multiplicative structure of $\Gamma_R$ is generated by the following multiplication table.
\begin{table}[H]
\centering
\begin{tabular}{|l|l|l|l|l|l|l|}
\hline
            & $x_i$ & $x_j$ & $x_k$     & $x_{i+j}$ & $x_{j+k}$ & $x_{i+j+k}$ \\ \hline
$x_i$       & $x_i$ & $x_j$ & $0$       & $x_{i+j}$ & $x_j$     & $x_{i+j}$   \\ \hline
$x_j$       & $0$   & $0$   & $x_j$     & $0$       & $x_j$     & $x_j$       \\ \hline
$x_k$       & $0$   & $0$   & $x_k$     & $0$       & $x_k$     & $x_k$       \\ \hline
$x_{i+j}$   & $x_i$ & $x_j$ & $x_j$     & $x_{i+j}$ & $0$       & $x_{i}$     \\ \hline
$x_{j+k}$   & $0$   & $0$   & $x_{j+k}$ & $0$       & $x_{j+k}$ & $x_{j+k}$   \\ \hline
$x_{i+j+k}$ & $x_i$ & $x_j$ & $x_{j+k}$ & $x_{i+j}$ & $x_k$     & $1$         \\ \hline
\end{tabular}
\end{table}

The additive structure has the following diagramatic form.
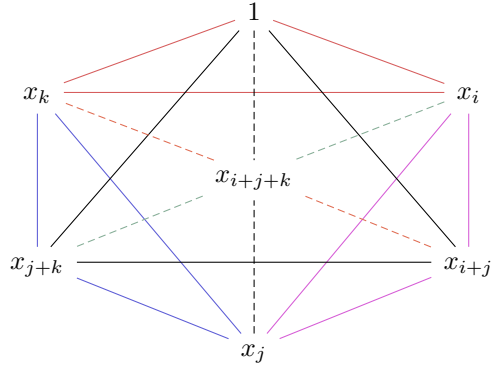
\[\begin{tikzcd}
	&& 1 \\
	{x_k} &&&& {x_i} \\
	&& {x_{i+j+k}} \\
	{x_{j+k}} &&&& {x_{i+j}} \\
	&& {x_j}
	\arrow[color={rgb,255:red,214;green,92;blue,92}, no head, from=2-1, to=2-5]
	\arrow[color={rgb,255:red,214;green,92;blue,214}, no head, from=2-5, to=5-3]
	\arrow[color={rgb,255:red,92;green,92;blue,214}, no head, from=2-1, to=5-3]
	\arrow[color={rgb,255:red,92;green,92;blue,214}, no head, from=2-1, to=4-1]
	\arrow[color={rgb,255:red,92;green,92;blue,214}, no head, from=5-3, to=4-1]
	\arrow[color={rgb,255:red,214;green,92;blue,214}, no head, from=2-5, to=4-5]
	\arrow[color={rgb,255:red,214;green,92;blue,214}, no head, from=5-3, to=4-5]
	\arrow[color={rgb,255:red,214;green,92;blue,92}, no head, from=1-3, to=2-1]
	\arrow[color={rgb,255:red,214;green,92;blue,92}, no head, from=1-3, to=2-5]
	\arrow[no head, from=4-5, to=1-3]
	\arrow[no head, from=4-1, to=1-3]
	\arrow[no head, from=4-1, to=4-5]
	\arrow[dashed, no head, from=1-3, to=3-3]
	\arrow[dashed, no head, from=5-3, to=3-3]
	\arrow[color={rgb,255:red,133;green,173;blue,150}, dashed, no head, from=2-5, to=3-3]
	\arrow[color={rgb,255:red,133;green,173;blue,150}, dashed, no head, from=3-3, to=4-1]
	\arrow[color={rgb,255:red,224;green,108;blue,82}, dashed, no head, from=3-3, to=4-5]
	\arrow[color={rgb,255:red,224;green,108;blue,82}, dashed, no head, from=3-3, to=2-1]
\end{tikzcd}\]

The sum of any two elements with the same lines type and color between them are the same item. So, for instance, we have
\[[x_i + x_k] = [x_k + 1] = [1 + x_i]\]
and
\[[1 + x_{i + j + k}] = [x_{i + j + k} + x_j] = [1 + x_j].\]

There is a single element corresponding to a non-degenerate sum of three elements, and it is obtained by the sum of any three elements such that the lines between them are distinct colors or types, e.g.:
\[[x_i + x_j + x_k] = [1 + x_{i + j} + x_j]\]
But, for instance, the element $[1 + x_{i + j} + x_{j + k}]$ is equal to $[1 + x_{i + j}]$, but is not equal to $[x_i + x_j + x_k]$.
\end{example}

Theorem 8 in \cite{bernardoni2023applications} classifies the additive structure of these universal valuation semirings.
\begin{theorem}[Structure Theorem for $\Gamma_R$]\label{thm:valladdstruct}\index{Structure Theorem for Universal Valuations}
    Let $(a_i)_{i \in I}$ and $(b_j)_{j \in J}$ be finite collections of elements in a ring $R$. In $\Gamma_R$ we have:
    \[\left[\sum_{i \in I}x_{a_i}\right] = \left[\sum_{j \in J} x_{b_j}\right]\]
    if and only if $\Span_{\Z}(a_i) = \Span_{\Z}(b_j)$.
\end{theorem}

\begin{remark}
    This theorem can be phrased another way.

    Consider $R$ as a $\Z$-module, and let $\operatorname{SubMod}_{\Z}(R)$ be the collection of $\Z$ submodules of $R$. We note that if $A_i$ is an arbitrary collection of $\Z$-submodules of $R$ then $\bigcap A_i \in \operatorname{SubMod}_{\Z}(R)$. From this we can define that for $X \subseteq R$ the \textbf{submodule generated by $X$} is:
    \[\gen{X} = \bigcap_{X \subseteq A \in \operatorname{SubMod}_{\Z}(R)} A\]

    We then note that $\operatorname{SubMod}_{\Z}(R)$ has an idempotent semiring structure, with:
    \begin{align*}
        A +_{\operatorname{SubMod}_{\Z}(R)} B &= \gen{A \cup B}\\
        A *_{\operatorname{SubMod}_{\Z}(R)} B &= \gen{ab : a \in A, b \in B}\\
        0_{\operatorname{SubMod}_{\Z}(R)} &= \gen{0_R} = \{0_R\}\\
        1_{\operatorname{SubMod}_{\Z}(R)} &= \gen{1_R}
    \end{align*}

    Let $\operatorname{SubMod}^0_{\Z}(R)$ be the collection of \textit{finitely generated} $\Z$ submodules of $R$. We note that this is a subsemiring of $\operatorname{SubMod}_{\Z}(R)$. There is a surjective valuation $R \to \operatorname{SubMod}^0_{\Z}(R)$ sending $r$ to the subspace generated by $r$, $\gen{r}$. Theorem \ref{thm:valladdstruct} states that the induced factoring map between $\Gamma_R$ and $\operatorname{SubMod}^0_{\Z}(R)$ is injective as well as surjective, and so $\Gamma_R \cong \operatorname{SubMod}^0_{\Z}(R)$. We further note that if we take our universal valuation to be the limit over valuations whose codomain are \textit{complete} idempotent semirings, then our universal object would be $\operatorname{SubMod}_{\Z}(R)$.
\end{remark}

In the remainder of this paper we will classify the universal object associated with two more classes of valuations: positive valuations and positive multiplicatively idempotent valuations. In the same way that the universal valuation semiring defines $\operatorname{SubMod}_{\Z}(R)$ up to isomorphism, we will see that these two additional classes of valuations define the semirings of ideals of a ring, as well as the closed Zariski topology on the spectrum of a ring up to isomorphism.

\section{Positive Valuations}\label{sec:PostiveVal}

\begin{conventions}
    While all of the results in \cite{bernardoni2023applications} hold for commutative and non-commutative rings, for the rest of this paper we will restrict to commutative rings.

    We will be particularly interested in the product of principally generated ideals, and for noncommutative rings it is not in general true that $(a)*(b) = (ab)$. As a result of this the map $r \mapsto (r)$ is not guaranteed to be multiplicative in the noncommutative case.
\end{conventions}

\begin{definition}\label{def:PositiveVal}
    A valuation $\nu : R \to \Gamma$ is \textbf{positive}\index{Valuation!Arithmetic!Positive} if for each $r \in R$, $\nu(r) \ge 1_{\Gamma}$
\end{definition}
\begin{remark}
    When $\Gamma = \mathbb T$, this is saying that for all $r \in R$, $\nu(r) \ge 0 \in \mathbb R$.
\end{remark}

\begin{proposition}
    Let $\nu : R \to \Gamma$ be any valuation. The set $\{r \in R : \nu(r) \ge 1_{\Gamma}\}$ is a subring of $R$.
\end{proposition}

\begin{definition}
    For a given valuation $\nu: R \to \Gamma$ we call the subring $\{r \in R : \nu(r) \ge 1_{\Gamma}\}$ the \textbf{$\nu$-positive subring of $R$}. When $\nu$ is understood we denote this $R_+$.
\end{definition}

\begin{remark}
    $R_+$ is the largest subring of $R$ where $\nu$ forms a positive valuation.
\end{remark}

\begin{proposition}
    The $\nu$-positive subring of the universal valuation semiring $\Gamma_R$ is exactly the image of the canonical map $\mathbb Z \to R$, sending $1_{\mathbb Z}$ to $1_R$.
\end{proposition}
\begin{proof}
    This is a consequence of our structure theorem, theorem \ref{thm:valladdstruct}, as if $x_r \ge 1 = x_1$ then $r \in \Span_{\mathbb Z}(1_R)$.
\end{proof}

It is well known classically that in certain contexts the existence of positive valuations are equivalent to deep statements on the structure of the ideals of a commutative ring.

\begin{definition}
    Let $R$ be a integral domain and let $K = \operatorname{Frac}(R)$.

    If there exists a valuation $\nu : K \to \Gamma$ where $\Gamma$ is a totally ordered idempotent semiring and $K_+ = R$, then we say that \textbf{$R$ is a valuation ring.}
\end{definition}

A classical result states:
\begin{theorem}
    If a domain is a valuation ring then its ideals are totally ordered by inclusion.
\end{theorem}
\begin{proof}
    We reproduce the classical approach. 
    
    Given a totally ordered valuation $\nu : K \to \Gamma$ then for all $x, y \in K$ either $xy^{-1} \ge 1$ giving us $xy^{-1} \in K_+ = R$ or $xy^{-1} \le 1$ giving us $yx^{-1} \in R$. If $xy^{-1} \in R$ then $(y) \subseteq (x)$ and if $y x^{-1} \in R$ then $(x) \subseteq (y)$. If the principal ideals are totally ordered then one can show that arbitrary ideals are totally ordered.
\end{proof}

This converse of this statement is also true and is a theorem by Krull, \cite{Krull1939}.

\begin{remark}
    For the above proof to work we just needed that for all $x, y \in K$ either $\nu(xy^{-1}) \le 1_{\Gamma}$ or $\nu(xy^{-1}) \ge 1_{\Gamma}$. The fact that $\Gamma$ is totally ordered gives us this, but we can generalize this considerably.
\end{remark}

\begin{proposition}
    Let $R$ be an integral domain and $K = \operatorname{Frac}(R)$.

    If there exists a valuation $\nu: K \to \Gamma$ such that for all $k \in K$, $\nu(k)$ is comparable to $1_{\Gamma}$, and $K_+ = R$, then the ideals of $R$ are totally ordered by inclusion.
\end{proposition}

While this result seems to be algebraic trickery, it suggests that there is a relationship between the space of positive valuations and the structure of the ideals of a semiring. This connection turns out to be a deep, and arguably defining, connection. We will formalize this connection, showing that every positive valuation (with complete codomain) admits a Galois connection with the lattice of ideals of a ring $R$, as well as showing that the semiring of ideals of a ring $R$ can be identified as the universal positive valuation that factors all other positive valuations with complete codomain.

Just as there is a universal valuation semiring, there is a universal \textit{positive} valuation semiring that factors all positive valuations.

\begin{definition}
    Given a ring $R$, the \textbf{universal positive valuation semiring}\index{Universal Positive Valuation}, $\Gamma_R^+$\index{$\Gamma_R^+$} is the semiring:
    \[\Gamma_R/\gen{x_r + 1 = 1}.\]
\end{definition}

\begin{proposition}
    This semiring, equipped with the valuation $\nu(r) = [r]$, is initial among the positive valuations of $R$.
\end{proposition}

\begin{definition}
    Given a (commutative or non-commutative) ring $R$, the set of two sided ideals of $R$, $I(R)$ has an idempotent semiring structure, with
    \[I + J = \{a + b : a \in I, b \in J\},\]
    \[I * J = \left(\sum_{i = 0}^N a_i b_i : a_i \in I, b_i \in J, N \in \mathbb N\right).\]

    In general $\sum I_j$ is the ideal generated by finite sums of elements in the indexed ideals, and $\prod_{j = 1}^N I_j$ is the ideal generated by finite sums of elements of the form $a_1*a_2*...*a_N$ for $a_j \in I_j$.
    
    Under this structure we have:
    \[1_{I(R)} = (1), \quad 0_{I(R)} = (0)\]
    \[I \le J \iff I \supseteq J.\]

    We call this the \textbf{semiring of ideals of $R$}\index{Semiring!Semiring of Ideals}.
\end{definition}

\begin{definition}
    When $R$ is commutative we may form a positive valuation $\iota: R \to I(R)$ under the mapping
    \[\iota : x \mapsto (x).\]
\end{definition}

\begin{theorem}\label{thm:GaloisConnection}
    Let $\nu : R \to \Gamma$ be a positive valuation, and let $\Gamma$ be a complete semiring\footnote{Or at least closed under sums of cardinality $|R|$ or less.}. Then there exists a monotone Galois connection between $I(R)$ and $\Gamma$.

    That is, there exist monotone maps $N : I(R) \to \Gamma$ and $I: \Gamma \to I(R)$ such that
    \[\gamma \le N(J) \iff I(\gamma) \le J\]
    for $\gamma \in \Gamma$, $J \in I(R)$.
\end{theorem}
\begin{proof}
    For each $\gamma \in \Gamma$ we first check that the set $I(\gamma) = \{r \in R : \nu(r) \ge \gamma\}$ forms an ideal.
    \begin{itemize}
        \item For all $\gamma$, $0_R \in I(\gamma)$.
        \item If $a, b \in I(\gamma)$ then $\nu(a - b) \ge \nu(a) + \nu(b) \ge \gamma$, and so $a - b \in I(\gamma)$.
        \item Let $a \in I(\gamma)$ and $r \in R$. $\nu(ra) = \nu(r) \nu(a) \ge 1_{\Gamma}*\gamma = \gamma$ and so $ra$ and similarly $ar$ are in $I(\gamma)$.
    \end{itemize}

    The map $\gamma \mapsto I(\gamma)$ is not quite a semiring homomorphism as it is not necessarily additive, but it is a monic map and if $R$ is a PID or a commutative Noetherian ring then it is guaranteed to be multiplicative.

    Similarly for any ideal $I$, as $\Gamma$ is complete we can define the map $N$ by $N(I) = \sum_{x \in I} \nu(x)$.
    
    We finally prove that $\gamma \le N(J)$ if and only if $I(\gamma) \le J$.

    Let $\gamma \le N(J)$, then for all $x \in J$ we have that $\nu(x) \ge \gamma$ and so $x \in I(\gamma)$. Therefore $I(\gamma) \supseteq J$ or $I(\gamma) \le J$ by Definition \ref{def:SemiringOrder}.

    Let $I(\gamma) \le J$, or similarly $J \subseteq I(\gamma)$, then for all $x \in J$ we have that $\nu(x) \ge \gamma$, and so $\sum_{x \in J} \nu(x) \ge \sum_{x \in J} \gamma = \gamma$.
\end{proof}
\begin{remark}
    This Galois connection means we can form an idempotent closure operator on $\Gamma$ by
    \[\gamma \mapsto N(I(\gamma)).\]
    Note that $\gamma \le N(I(\gamma))$ always holds.

    We can also form an idempotent kernel operator on $I(R)$ by
    \[J \mapsto I(N(J))\]
    We similarly also have that $I(N(J)) \le J$ or $I(N(J)) \supseteq J$.
\end{remark}

\begin{example}
    Consider $\Z$ equipped with the $p$-adic valuation.

    $\Z$ is a PID, and we get that for $(x) \subseteq Z$, $N((x)) = \nu(x)$.

    Our kernel operator is $I(N(x)) = (p^{\nu(x)})$, our closure operator $N(I(a)) = N((p^a)) = a$ for all $a$. Under the $3$-adic valuation we get that $I(N((2)) = \Z$, but $I(N((6))) = (3) = I(N((24)))$.
\end{example}

\begin{theorem}
    For any positive valuation $\nu: R \to \Gamma$ for which $\Gamma$ is complete, the induced map $N: I(R) \to \Gamma$ is a complete semiring homomorphism whuch factors $\nu$. That is $\nu = N \circ \iota$, where $\iota$ is the valuation $\iota(r) = (r)$ on $I(R)$.
\end{theorem}

\begin{proof}
    While the map $I$ does not always form a semiring homomorphism, we can see that $N$ does, as
    \[N(I + J) = \sum_{x + y, x \in I, j \in J} \nu(x + y) \ge \sum_{x + y, x \in I, j \in J} \nu(x) + \nu(y) = \sum_{x \in I} \nu(x) + \sum_{y \in J} \nu(y).\]
    As $0 \in I \cap J$, and we have that in an idempotent semiring sums are always less than or equal to subsums, we have
    \[\sum_{x + y, x \in I, j \in J} \nu(x + y) \le \sum_{x \in I} \nu(x)\]
    and
    \[\sum_{x + y, x \in I, j \in J} \nu(x + y) \le \sum_{y \in J} \nu(y)\]
    and so due to additive idempotence we get
    \[\sum_{x + y, x \in I, j \in J} \nu(x + y) \le \sum_{x \in I} \nu(x) + \sum_{y \in J} \nu(y)\]
    and so $N$ forms an additive homomorphism. By finite superadditivity and the fact that for an infinite sum, elements in $\sum J_i$ are finite sums of elements in $J_i$'s, we can see that $N$ also respects the infinitary sums of $I(R)$.

    We have that by superadditivity 
    \[N(I * J) \ge \sum_{a \in I, b \in J} \nu(a)\nu(b) = \sum_{a \in I}\nu(a) * \sum_{b \in J} \nu(b),\] however as $I*J \supseteq \{a b : a \in I, b \in J\}$ we have that $\sum_{a \in I}\nu(a) * \sum_{b \in J} \nu(b)$ is a subsum of $N(I * J)$ and thus is greater than or equal to it. Thus $N(I*J) = N(I)*N(J)$.

    $N$ forms a map of positive valuations. Let $\nu: R \to \Gamma$ be a positive valuation. As $\nu(r) \ge 1_{\Gamma}$ for all $r \in \Gamma$ then for any $x \in \Gamma$ we have $\nu(rx) = \nu(r)\nu(x) \ge \nu(x)$. Thus for any any principal ideal we have $N((x)) = \nu(x)$, and so $\nu = N\circ\iota$.
\end{proof}

\begin{remark}
    Once we know that the map $N: I(R) \to \Gamma$ exists it is not suprising that there \textit{exists} a Galois correspondence. Complete idempotent semiring homomorphisms are meet preserving maps from posets that have all meets, and so by the adjoint functor theorem for posets every complete idempotent semiring homomorphism admits an adjoint. The nLab article \cite{nlab:adjoint_functor_theorem} has more details on this construction.

    It is suprising however that the specific adjoint of $N: I(R) \to \Gamma$ is cleanly given by the preimage of the upward closure of an element in $\Gamma$: $I(\gamma) = \nu^{-1}\left(\{\rho : \rho \ge \gamma\}\right)$
\end{remark}

\begin{remark}
    We note that if $\{a_i\} \subseteq R$ generate the ideal $J$ then $N(J) = \sum a_i$.

    If $x \in J$ we can write it as a finite sum $x = \sum r_i a_i$. By superadditivity and positivity of $\nu$ we get 
    \[\nu(x) \ge \sum \nu(r_i)\nu(a_i) \ge \sum \nu(a_i).\]
    The behavior of $N$ is entirely determined by the valuation on generating sets of our ideals.
\end{remark}

\begin{corollary}\label{thm:I(R) initial}
    $I(R)$ is initial among positive valuations whose valuation semiring is complete for sums of size $|R|$ or less.\footnote{Strictly speaking: Let $\beth$ be the largest cardinality of a minimal generating set of an an ideal of $R$, we can define the map $N : I(R) \to \Gamma$ as long as $\Gamma$ is closed under sums of cardinality $\beth$ or less.}

    If $R$ is Noetherian then $I(R)$ is initial among all positive valuations.
\end{corollary}

Here we note that if $R$ is Noetherian then every ideal is finitely generated, and so we can reduce every sum, $N(I)$, to a finite sum.

\begin{corollary}
    If $R$ is Noetherian, then $I(R) \cong \Gamma_R^+ = \Gamma_R/\gen{x_r + 1 = 1}$
\end{corollary}

\begin{example}
    We will briefly show using example \ref{ex:SmallestNC} why our assumption of commutativity is necessary. Here $R$ is the subring of $M_2(\Z_2)$ consisting of upper triangular matrices. We will name elements of $R$ as:
    \[i = \begin{bmatrix}1&0\\0&0\end{bmatrix} \qquad j = \begin{bmatrix}0&1\\0&0\end{bmatrix} \qquad k = \begin{bmatrix}0&0\\0&1\end{bmatrix}\]

    Here we note that $R$ is finite and every semiring is closed under finite sums, and so if our theorem held without requiring commutativity we would have $\Gamma_R^+ \cong I(R)$.

    The elements of $I(R)$ are the two sided ideals:
    \begin{align*}
        (0) &= \{0\}\\
        (i) = (i + j) &= \{0, i, j, i + j\}\\
        (j) &= \{0, j\}\\
        (k) = (j + k) &= \{0, j, k, j + k\}\\
        (1) = (i + k) = (i + j + k) &= R
    \end{align*}

    The multiplication table of $I(R)$ is the following.
\[
\begin{tabular}{|l|l|l|l|}
\hline
      & $(i)$ & $(j)$ & $(k)$ \\ \hline
$(i)$ & $(i)$ & $(j)$ & $(j)$ \\ \hline
$(j)$ & $(0)$ & $(0)$ & $(j)$ \\ \hline
$(k)$ & $(0)$ & $(0)$ & $(k)$ \\ \hline
\end{tabular}\]

    And the additive structure is given by the following lattice, with an arrow $a \to b$ meaning $a \le b$.
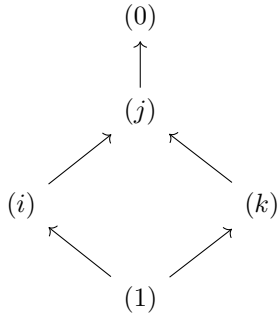
\[\begin{tikzcd}
	& {(0)} \\
	& {(j)} \\
	{(i)} && {(k)} \\
	& {(1)}
	\arrow[from=4-2, to=3-1]
	\arrow[from=4-2, to=3-3]
	\arrow[from=3-3, to=2-2]
	\arrow[from=3-1, to=2-2]
	\arrow[from=2-2, to=1-2]
\end{tikzcd}\]

    We note that this is indeed the structure one would expect when quotienting just the additive structure of $\Gamma_R$ from Example \ref{ex:SmallestNC} by the congruence generated by $x_r + 1 = 1$ for all $r \in R$. However this does not form a valuation as the map $r \mapsto (r)$ is not multiplicative: $i *_R k = 0$ but $(i) *_{I(R)} (k) = (j)$. The initial positive valuation of $R$ in this case is not $I(R)$ but $I(R)/\gen{(j) \sim 0}$.
\end{example}

A consequence of Corollary \ref{thm:I(R) initial} is that the complete idempotent semiring of ideals of a ring can be characterized up to isomorphism via the category of positive $R$-valuations whose codomain is complete. Many properties of the ideals of $R$ can be extracted from the structure of this semiring, even without reference to the underlying sets. For instance one can derive the collection of prime ideals of a ring purely from the semiring structure.

\begin{definition}
    Let $\Gamma$ be an idempotent semiring.

    We say $p \in \Gamma$ is a \textbf{prime element} if $ab \ge p$ implies $a \ge p$ or $b \ge p$.
\end{definition}

\begin{proposition}\label{thm:KrullsDef}
    $J \in I(R)$ is a prime element if and only if $J$ is a prime ideal.
\end{proposition}
Proposition \ref{thm:KrullsDef} is exactly Krull's definition of a prime ideal for a possibly noncommutative ring.

\begin{definition}
    Let $\Gamma$ be an idempotent semiring. We say $F \subseteq \Gamma$ is a \textbf{filter} if it satisfies the following
    \begin{itemize}
        \item $F$ is nonempty.
        \item $F$ is upward closed: for each $x \in F$, $y \in \Gamma$ $x \le y$ implies that $y \in F$.
        \item $F$ is downward directed: for every $x, y \in F$, there is some $z \in F$ such that $x \ge z$ and $y \ge z$.
    \end{itemize}

    This third point in particular implies that for any $x, y \in F$ we have $x + y \in F$.

    A filter is \textbf{prime} if $xy \in F$ implies $x \in F$ or $y \in F$.

    A filter is \textbf{radical} if it can be written as an intersection of prime filters.

    A filter is \textbf{principal} if it is of the form $\overline{x} = \{y \in \Gamma : y \ge x\}$ for some $x \in \Gamma$.

    An filter is \textbf{maximal} if it is proper and there is no proper filter containing it. Maximal filters are also called \textbf{ultrafilters}.
\end{definition}

These definitions, with the exception of primality, are the standard order theoretic definitions. When our idempotent semiring is a distributive lattice -- meaning $ab = a \vee b$ -- our definition of a prime element and a prime filter are exactly the standard ones.

\begin{proposition}\label{thm:PrincipalPrimeisPrime}
    A principal filter $\overline{x}$ is prime if and only if $x$ is a prime element.
\end{proposition}

\begin{proposition}
    $J \in I(R)$ is a radical ideal if and only if $\overline{J}$ is a radical filter.
\end{proposition}
We note here that with Proposition \ref{thm:PrincipalPrimeisPrime}, the fact that the principal filter of radical ideal is radical may seem trivial, the converse is not necessarily simple. It is important to note that a prime filter is not necessarily a prime \textit{principal} filter, and so the statement that $\overline{J}$ is a radical filter does not directly state that $J$ is the intersection of prime ideals.
\begin{proof}
    Let $J = \sqrt{J}$ then we may write $J = \bigcap P_i$. We note that if $K \in I(R)$ is such that $K \subseteq J$ (i.e. $K \ge J$) then $K \subseteq P_i$ for all $i \in I$, and so $K$ is contained in the principal prime filter $\overline{P_i}$ for all $i$. Similarly if $K \in \overline{P_i}$ for all $i$ then $K \subseteq P_i$ for all $i$ and so $K \subseteq \bigcap P_i = J$, Thus we get $\overline{J} = \bigcap_{i \in I} \overline{P_i}$.

    Let $\overline{J} = \bigcap_{i \in I} F_i$ for $F_i$ prime filters. Then we note that $J \in F_i$ for all $i$, and for all $x \in J$ we have $(x) \in F_i$ for all $i$. Let $a^n \in J$, then $(a^n) = (a^{n-1})*(a) \in F_i$ for all $i$. As these are prime we get that $(a) \in F_i$ for all $i$, and so $(a) \in \overline{J}$. Thus we have $(a) \ge J$ or $(a) \subseteq J$, thus $a \in J$ and we know that $J$ is radical.
\end{proof}

Primality of elements and the above properties of filters are invariant under semiring isomorphisms, which means that we can reconstruct the prime spectrum and the Zariski topology of a commutative ring purely from the semiring resulting from the universal property given in Corollary \ref{thm:I(R) initial}.

\section{Positive, Multiplicatively Idempotent Valuations}\label{sec:FrameVal}

What if we wanted to assign a \textit{set} or some lattice valued structure to each element of a ring?

\begin{example}
    Let $R = k[x_1, ..., x_n]$ for $k$ an algebraically closed field.

    Consider the map:
    \[f \mapsto \{\vec{x} : f(\vec{x}) = 0_k\}\]

    This forms a positive, multiplicatively idempotent valuation into the semiring $\mathcal P(k^n)$, where for $X, Y \subseteq k^n$ we have:
    \begin{align*}
        X + Y &= X \cap Y,\\
        X * Y &= X \cup Y,\\
        X \le Y &\iff X \subseteq Y,\\
        0 &= k^n,\\
        1 &= \emptyset.
    \end{align*}

    The induced Galois correspondence $I(R) \vdash \mathcal P(k^n)$ is given by the maps below.
    \begin{align*}
        N : I(R) &\to \mathcal P(k^n)\\
        N: J &\mapsto \bigcap_{j \in J}\{\vec{x} : j(\vec{x}) = 0\}\\
        I: \mathcal P(k^n) &\to I(R)\\
        I: X &\mapsto \{f : f(\vec{x}) = 0 \, \forall x \in X\}\\
        ~\\
        IN(J) &= \sqrt{J}
    \end{align*}

    This is exactly the standard Galois correspondence in algebraic geometry between ideals and varieties. The statement that $IN(J) = \sqrt{J}$ is famously known as Hilbert's Nullstellensatz.
\end{example}

    We might ask ourselves if Hilbert's Nullstellensatz is a sort of universal result, or arises out of some universal result on positive multiplicatively idempotent valuations. We will show that this radical closure is the finest closure that can arise out of a positive multiplicatively idempotent valuation.

    \begin{definition}\label{def:MultIdempVal}
        Let $\nu: R \to \Gamma$ be a positive valuation. We say that $\nu$ is \textbf{multiplicatively idempotent}\index{Valuation!Arithmetic!Multiplicatively Idempotent} if $\Gamma$ is, that is, if for all $\gamma \in \Gamma$ we have
        \[\gamma *_{\Gamma} \gamma = \gamma.\]
    \end{definition}
    
    \begin{remark}
        A positive multiplicatively idempotent valuation is equivalent to a \textit{notion de z\'eros} introduced by Andr\'e Joyal in \cite{Joyal}. In this note he states that the prime spectrum of a ring factors all notions of zeroes. We will make this statement precise. An explanation of the consequences of this statement can be found in a Mathematics Stack Exchange answer by Zhen Lin Low, \cite{Zhen}.
    
        To make this statement precise we need to restrict to the case where $I(R)$ is initial, see Corollary \ref{thm:I(R) initial}. For this to hold we either consider slices over the category of \textit{$|R|$-complete} idempotent semirings or we assume that $R$ is Noetherian.
    \end{remark}

    \begin{remark}
        In the field of point-free topology, a complete, commutative, multiplicatively and additively idempotent semiring has a punchier name, it is a \textbf{coframe}. Coframes can be thought of as a generalization of closed topologies -- every closed topology is a coframe.
        
        Our positive, multiplicatively idempotent valuations can be seen as finding maps $\nu : R \to \Gamma$ from a commutative ring $R$ to a coframe or closed topology $\Gamma$ such that
        \begin{align*}
            \nu(0) &= \top,\\
            \nu(1) &= \bot,\\
            \nu(a + b) &\supseteq \nu(a) \wedge \nu(b),\\
            \nu(ab) &= \nu(a) \vee \nu(b).
        \end{align*}

        If $\Gamma$ is a topology $\top$ is the entire set and $\bot = \emptyset$.
    \end{remark}

    \begin{proposition}\label{thm:Closure contains radical}
        Let $R \overset{\nu}{\to} \Gamma$ be a positive, multiplicatively idempotent valuation. For any ideal $J \subseteq R$ we get:
        \[IN(J) \supseteq \sqrt{J}\]
    \end{proposition}

    \begin{proof}
        As $\Gamma$ is multiplicatively idempotent we have that $\nu(f) = \nu(f^n)$ for all $n$. Then $\nu(f) \ge \gamma$ if and only if for any $n$ we have $\nu(f^n) \ge \gamma$.

        If $r^n \in I(\gamma)$, then $\nu(r^n) = \nu(r) \ge \gamma$ and so $r \in I(\gamma)$.
    \end{proof}

    For any commutative ring any positive multiplicatively idempotent valuation gives a closure operator on the ideals of that ring that is ``lower bounded by the Nullstellensatz". Our question is then, can we obtain a correspondence whose closed ideals are exactly the radical ideals? To do so we will classify the universal positive multiplicatively idempotent valuation semiring.

    In the same way that we can construct a universal valuation semiring, or a universal positive valuation semiring, we can construct a universal positive multiplicatively idempotent semiring.

    \begin{definition}
        The \textbf{universal positive multiplicatively idempotent valuation semiring}\index{Universal Positive Multiplicatively Idempotent Valuation}, $\Gamma^{id}_R$\index{$\Gamma^{id}_R$}, is the semiring
        \[\Gamma_R^+/\gen{x_r^2 = x_r} \cong \Gamma_R/\gen{x_r + 1 = 1, x_r^2 = x_r}\]
        equipped with the valuation
        \[r \mapsto [x_r].\]
    \end{definition}

    \begin{theorem}
        This valuation is initial among the positive, multiplicatively idempotent valuations of $R$.
    \end{theorem}

    \begin{remark}
        When $I(R)$ is initial among positive valuations our initial positive multiplicatively idempotent valuation semiring is the semiring $I(R) / \gen{(r)*(r) = (r)}.$
    \end{remark}

    \begin{theorem}\label{thm:Closed correspondence is radical}
        The closed ideals of the correspondence between $I(R)$ and $\Gamma^{id}_R$ are exactly the radical congruences\footnote{In the case where $R$ is not Noetherian, we look at the correspondence between $I(R)$ and the initial valuation among the valuations whose codomains are complete semirings.}.
    \end{theorem}

    To prove this it will suffice to find a single valuation that distinguishes between radicals. We will do that first and then present the proof of this statement.

    \begin{example}\label{ex: Spec valuation}
        For a commutative ring $R$, we have a valuation $\tau : R \mapsto T$ where $T$ is the closed Zariski topology on $\operatorname{Spec}(R)$ with
        \begin{align*}
            A +_{T} B &= A \cap B,\\
            A *_{T} B &= A \cup B,\\
            0_{T} &= \operatorname{Spec}(R),\\
            1_{T} &= \emptyset.
        \end{align*}
        $T$ is a complete idempotent semiring.

        The associated valuation is the assignment of distinguished closed sets given below.
        \[\tau : f \mapsto D_f = \{P \text{ prime} : f \in P\}\]

        Our correspondence is given by the maps:
        \begin{align*}
            N: I(R) & \to T\\
            N: J &\mapsto \bigcap_{j \in J} D_j = \{P \text{ prime}: P \supseteq J\}\\
            I: T &\to I(R)\\
            I: X &\mapsto \bigcap_{P \in X}P
        \end{align*}

        We can see that $IN(J)$ is the intersection of prime ideals containing $J$, so $IN(J) = \sqrt{J}$.
    \end{example}

    We can now prove Theorem \ref{thm:Closed correspondence is radical}.
    \begin{proof}
        We want to show that under the induced Galois correspondence between $I(R)$ and $\Gamma^{id}_R$ we have that $IN(J) = \sqrt{J}$. We know from Proposition \ref{thm:Closure contains radical} that $IN(J) \supseteq \sqrt{J}$, so it remains to show that $IN(J) \subseteq \sqrt{J}$. To do so we will show that the closure given by the valuation $\tau$ gives an upper bound to the closure on the universal valuation.

        Let $\tau : R \to T$ be the valuation in Example \ref{ex: Spec valuation}, and let $f$ be the unique map from $\Gamma_{R}^{id}$ to $T$ that makes the following diagram commute:
\[\begin{tikzcd}
	& R \\
	{\Gamma^{id}_R} && T
	\arrow["f"{description}, from=2-1, to=2-3]
	\arrow["\nu"{description}, from=1-2, to=2-1]
	\arrow["\tau"{description}, from=1-2, to=2-3]
\end{tikzcd}\]
        Let $N_{\Gamma}$ be the correspondence map $I(R) \to \Gamma^{id}_R$ and $N_{T}$ be the correspondence map $I(R) \to T$. Similarly $I_{\Gamma}$ the map $\Gamma^{id}_R \to I(R)$ and $I_{T}$ the map $T \to I(R)$

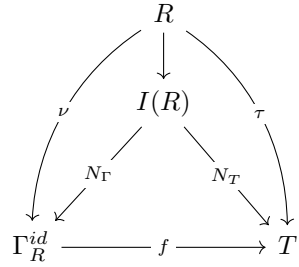
        \[\begin{tikzcd}
        	& R \\
        	& {I(R)} \\
        	\\
        	{\Gamma^{id}_R} && T
        	\arrow[from=1-2, to=2-2]
        	\arrow["\nu"{description}, curve={height=18pt}, from=1-2, to=4-1]
        	\arrow["\tau"{description}, curve={height=-18pt}, from=1-2, to=4-3]
        	\arrow["{N_{\Gamma}}"{description}, from=2-2, to=4-1]
        	\arrow["{N_{T}}"{description}, from=2-2, to=4-3]
        	\arrow["f"{description}, from=4-1, to=4-3]
        \end{tikzcd}\]

        From the fact that $\tau = \nu \circ f$ and that $I_T(x) = \tau^{-1}(\{t \in T: t \ge x\}$ we get the following statements.
        \[I_{T}(x) = \tau^{-1}(\{t \in T : t \ge f(x)\}) = \nu^{-1}(f^{-1}(\{t \in T : t \ge f(x)))\]
        \[ \nu^{-1}(\{\gamma \in \Gamma^{id}_R : \gamma \ge x\}) = I_{\Gamma}(x)\]
        We note that $f^{-1}(\{t \in T : t \ge f(x)\}) \supseteq \{\gamma \in \Gamma^{id}_R : \gamma \ge x\}$ and so we get:
        \[I_T(x) = \nu^{-1}(f^{-1}(\{t \in T : t \ge f(x))) \supseteq \nu^{-1}(\{\gamma \in \Gamma^{id}_R : \gamma \ge x\}) = I_{\Gamma}(x).\]
        Put succinctly, $I_{T}(f(x)) \supseteq I_{\Gamma}(x)$

        We note that as our $N_{\bullet}$ maps are unique maps from an initial object we have that $N_T = fN_{\Gamma}$. Combined with the statement that $I_T(f(x)) \supseteq I_{\Gamma}(x)$ we get our upper bound.
        
        \[\sqrt{J} = IN_{T}(J) = I(f(N_{\Gamma}(J))) \supseteq IN_{\Gamma}(J)\]
        
        Together with Proposition \ref{thm:Closure contains radical} we get:
        \[\sqrt{J} \supseteq IN_{\Gamma}(J) \supseteq \sqrt{J}\]
        \[IN_{\Gamma}(J) = \sqrt{J}\]
    \end{proof}

    The valuation $\tau: R \to T$ can be used for more than bounding the closure operator on the universal valuation. We can show that it \textit{is} the universal valuation.
    \begin{theorem}\label{thm:SpecisUniversal}
        The initial valuation among appropriately complete, multiplicatively idempotent positive valuations is $\tau : R \to T$.
    \end{theorem}
    \begin{proof}
        We note that by definition $N: I(R) \to T$ is surjective, as each element in $T$ is a Zariski closed set, and so is of the form
        \[N(J) = \{P \text{ prime} : P \supseteq J\}\]
        for some ideal $J$. We also note that the map $N$ factors through $I(R)/\gen{x^2 \sim x}$, as $T$ is multiplicatively idempotent.

        We then need to show that if $N(J_1) = N(J_2)$ then $J_1 \sim J_2$ under the equivalence relation $x^2 \sim x$, and we then get that the factored map $I(R)/\gen{x^2 \sim x} \to T$ is an isomorphism of appropriately complete multiplicatively and additively idempotent semirings.

        We first note that if $N(J_1) = N(J_2)$ then $\sqrt{J_1} = \sqrt{J_2}$. 
        
        Under the equivalence relation induced by $x^2 \sim x$ we have that $\sqrt{J} \sim J$.
        \begin{align*}
            \sqrt{J} = \sum_{\{r : r^n \in J\}}(r) \sim \sum_{\{r : r^n \in J\}}(r)^n = \sum_{r \in J}(r) = J\\
        \end{align*}
        If $N(J_1) = N(J_2)$ then $J_1 \sim \sqrt{J_1} = \sqrt{J_2} \sim J_2$ under the congruence generated by $x^2 \sim x$, and so the map $I(R)/\gen{x^2 \sim x} \to T$ is injective.

        Thus we get that the factored map $\Gamma^{id}_R \cong I(R)/\gen{x^2 = x} \to T$ is both injective and surjective, and so it is an isomorphism.
    \end{proof}

    \begin{remark}
        Any assignment of lattice or set to elements of $R$ that ``looks like" assigning zeros to polynomials, ``looks like" you are taking a quotient of the coframe $\operatorname{Spec}(R)$, as that assignment factors through $\operatorname{Spec}(R)$. While the zero locus itself is not necessarily the universal valuation in this setting, we have that the Zariski topology on $\operatorname{Spec}(R)$ is.
    \end{remark}

    \begin{corollary}
        Let $D: R \to \operatorname{Spec}(R)$ be the map:
        \[r \mapsto \{\mathfrak p : r \in \mathfrak p\}\]
    
        As a distributive, inf-complete lattice (i.e. a coframe), the closed topology $\operatorname{Spec}(R)$ with singleton closure given by $D$ is defined up to isomorphism by the universal property:

        For any positive, multiplicatively idempotent valuation $\nu: R \to \Gamma$, there exists a unique map $\hat{\nu}: \operatorname{Spec}(R) \to \Gamma$ such that $\nu = \hat{\nu} \circ D$.
    \end{corollary}

    We note that this universal property is the definition of the frame theoretic spectrum of a ring given in section V.3.1 of \cite{StoneSpaces}. Through the generalized definition of a valuation given by Jeffrey and Noah Giansiracusa in \cite{2016} we can link this frame theoretic universal property directly to the field of arithmetic valuation theory.

    \begin{figure}[]
\begin{tabular}{|l|l|l|l|}
\hline
                    & Valuations      & \begin{tabular}[c]{@{}l@{}}Positive (Complete)\\ Valuations\end{tabular} & \begin{tabular}[c]{@{}l@{}}Multiplicatively Idempotent\\ Positive (Complete) \\ Valuations\end{tabular} \\ \hline
Universal Valuation & $\Z$-submodules & The semiring of ideals                                                  & \begin{tabular}[c]{@{}l@{}}The coframe of the Zariski topology\\ on the prime spectrum\end{tabular}                   \\ \hline
\end{tabular}
\caption{Initial valuations of a ring $R$, if $R$ is non-Noetherian then the positive valuations are over \textit{complete} semirings.}
\end{figure}

\printbibliography

\end{document}